\documentclass[english]{amsart}
\usepackage{amsfonts, amsmath, amssymb, amscd, amsthm, array, graphicx}
\usepackage[utf8]{inputenc}
\usepackage{amsmath}
\usepackage{stmaryrd}
\usepackage{graphicx}
\usepackage{amssymb}
\usepackage{dsfont}
\usepackage{babel}
\usepackage{color}
\usepackage[colorlinks=true, linkcolor=blue, citecolor=blue, urlcolor=blue, breaklinks=true]{hyperref}
\usepackage[capitalise]{cleveref}
\usepackage[all,cmtip]{xy}
\usepackage{float}
\usepackage{tikz}
\usetikzlibrary{arrows,shapes,automata,backgrounds,petri,decorations}
\usetikzlibrary{automata} 
\usetikzlibrary[automata]
\usepackage{geometry}
\usepackage{pdflscape}
\usetikzlibrary{external,automata,trees,positioning,shadows,arrows,shapes.geometric}
\usepackage{tikz-cd}
\usetikzlibrary{decorations.pathreplacing,decorations.markings}
\tikzset{close/.style={near start,outer sep=-2pt}} 
\tikzset{
  on each segment/.style={
    decorate,
    decoration={
      show path construction,
      moveto code={},
      lineto code={
        \path [#1]
        (\tikzinputsegmentfirst) -- (\tikzinputsegmentlast);
      },
      curveto code={
        \path [#1] (\tikzinputsegmentfirst)
        .. controls
        (\tikzinputsegmentsupporta) and (\tikzinputsegmentsupportb)
        ..
        (\tikzinputsegmentlast);
      },
      closepath code={
        \path [#1]
        (\tikzinputsegmentfirst) -- (\tikzinputsegmentlast);
      },
    },
  },
  mid arrow/.style={postaction={decorate,decoration={
        markings,
        mark=at position .5 with {\arrow[#1]{stealth}}
      }}},
}

\usepackage{mathtools}
\usepackage{hyperref,url}

\newtheorem{thm}{Theorem}[section]
\newtheorem*{thm*}{Theorem}
\newtheorem{cor}[thm]{Corollary}

\theoremstyle{definition}
\newtheorem{defn}[thm]{Definition}
\newtheorem{exam}[thm]{Example}

\theoremstyle{definition}
\newtheorem{rem}[thm]{Remark}

\newtheorem{nt}[thm]{Note}
\numberwithin{equation}{section}

\newcommand{\im}{\operatorname{Im}}

\newcommand{\res}{\operatorname{res}}
\newcommand{\cors}{\operatorname{cor}}

\newcommand{\ZZ}{\mathbb{Z}}
\newcommand{\QQ}{\mathbb{Q}}

\begin{document}

\title{Bogomolov multipliers of word labelled oriented graph groups}

\author{Mallika Roy}
\address{Harish-Chandra Research Institute, HBNI, Chhatnag Road, Jhunsi, Prayagraj (Allahabad) 211019, India.}
\email{mallikaroy75@gmail.com}

\subjclass{20E05, 20F05, 20F67}
\keywords{word labelled oriented graph group, Bogomolov multiplier, Artin group, Bestvina--Brady group.}

\begin{abstract} 
A group, whose presentation is explicitly derived in a certain way from a word labelled oriented graph (in short, WLOG), is called a WLOG group. In this work, we study homological version of Bogomolov multiplier (denoted by $\widetilde{B_0}$) for this family of groups. We prove how to compute the generators for the $\widetilde{B_0}(G)$ of a WLOG group $G$ from
the underlying WLOG. We exhibit finitely presented Bestvina--Brady groups and Artin groups as WLOG groups. As applications, we compute both the multipliers: the homological version of Bogomolov multipliers and Schur multipliers, of these groups utilizing their respective WLOG group presentations. Our computation gives a new proof of the structure of the Schur multiplier of a finitely presented Bestvina--Brady group.
\end{abstract}
\maketitle

\section{Introduction}\label{sec: intro}

In this paper, we discuss a homological version of Bogomolov multiplier, denoted by $\widetilde{B}_0(G)$ for a word labelled oriented graph (a.k.a WLOG) group $G$. The Bogomolov multiplier is a group theoretical invariant, which is isomorphic to the unramified Brauer group of a given quotient space. In~\cite{Moravec}, Moravec derived the functor $\widetilde{B}_0(G)$,
proved a Hopf-type formula and deduced a five term exact sequence corresponding to this invariant. As applications, we effectively compute $\widetilde{B}_0(G)$ for all Artin groups and Bestvina--Brady groups. 

Artin--Tits groups were introduced by J. Tits in the 1960s as natural generalizations of braid groups found by E. Artin. This is a class of groups determined through their presentations and spans a wide range of groups from free groups to free abelian groups and braid groups as well connecting low dimensional topology and geometry. We first recall the definition of Artin groups. An \emph{Artin group} $A$ is a group with presentation of the form:

\[
A = \langle a_1, \ldots , a_n \, | \, \underset{m_{ij}}{\underbrace{a_ia_ja_i\ldots}} \, = \, \underset{m_{ji}}{\underbrace{a_ja_ia_j\ldots}\,}, \,\text{ for all } i \neq j \rangle,
\]

where $m_{ij} = m_{ji}$ is an integer $\geq$ $2$ or $m_{ij} \, = \, \infty$ in which case we omit the relation between $a_i$ and $a_j$. If $a_i, a_j$ are two letters and $m_{ij}$ is an integer $\geq$ $2$, let $\langle a_i, a_j \rangle^{m_{ij}}$ denotes the word $a_ia_ja_i \ldots$ of finite length $m_{ij}$. We set $\langle a_i, a_j \rangle^{m_{ij}} = (a_ia_j)^{m_{ij}/2}$ if $m_{ij}$ is \emph{even} and $\langle a_i, a_j \rangle^{m_{ij}} = (a_ia_j)^{(m_{ij}-1)/2}a_i = a_i(a_ja_i)^{(m_{ij}-1)/2}$ if $m_{ij}$ is \emph{odd}. An Artin group $A$ is \emph{even} if any finite $m_{ij}$ is even; and we denote an even Artin group by $A_\Gamma$. The subscript $\Gamma$ comes from its corresponding \emph{Artin-Tits system}. We also recall the definition of an Artin-Tits system to set our convention. Let $\Gamma$ be a finite simplicial graph with vertex set $V\Gamma$ and the edge set $E\Gamma$, i.e., $V\Gamma$ is non-empty and neither loop nor the multiple edge is allowed. Also, let $m \colon E\Gamma \longrightarrow  \{2, 3, \ldots \}$ be a function. The pair $\Gamma = \left((V\Gamma, E\Gamma), m \right)$ is so called
\emph{Artin–Tits system}, and $m$ is called the \emph{labeling} of $\Gamma$. Clearly, for an even Artin group $A_\Gamma$, the labeling of $\Gamma$ is $m \colon E\Gamma \longrightarrow  \{2, 4, 6, \ldots \}$.

A \emph{right-angled Artin group} (in short, RAAG) is one in which $m_{ij} \in \{ 2, \infty\}$ for all $i,j$. In other words, in the presentation for the Artin group, all relations are commutator relations: $a_i a_j \, = \, a_j a_i$. We will denote an arbitrary RAAG by $G_\Gamma$. The conventional way to specify the presentation for a right-angled Artin group is by means of the defining graph $\Gamma$ (see Definition~\ref{dfn: RAAG, BB}). Here we would like to highlight that Artin-Tits system $\Gamma$ and the defining graph $\Gamma$ of a RAAG are eventually same---we ignore the labelling $m$ as it is always $2$, also we remind that we omit the relation from the presentation when $m =\infty$. We note that any finite, simplicial graph $\Gamma$ is the defining graph for a RAAG $G_\Gamma$. Right-angled Artin groups were first introduced in the 1970’s by A. Baudisch~\cite{Baudisch} and further developed in the 1980’s by C. Droms \cite{Droms-coherence}, \cite{Droms-isomorphism}, \cite{Droms-subgroups} under the name ``graph groups''.  The study of RAAGs from different perspectives
has contributed to the development of new, rich theories such as the theory of cube complexes and has been an
essential ingredient in Agol's solution to the Virtually Fibered Conjecture. Moreover, the family of RAAGs stands out as a family of groups from the richness of its family of subgroups: one of such family of subgroups is Bestvina--Brady groups.

A word labelled oriented graph (WLOG) group is a group given by a presentation constructed in a certain way from a WLOG (see~\ref{def: log groups}). The main examples include Wirtinger presentations of knot groups, Braid groups, right-angled Artin groups (see~\cite{Howie}, \cite{N. D. Gilbert}). In this article, we study two families of groups, namely, Bestvina--Brady groups and Artin groups, as WLOG groups. Then, we avail the Bogomolov multipliers and Schur multipliers of the above families of groups via their presentations as WLOG groups. We also compute the Bogomolov multiplier for right-angled Artin groups; in fact, $\widetilde{B_0}(G_\Gamma)$ is trivial.

In this article, we consider the Schur multiplier of a group $G$, which is the second homology group $H_2(G, \ZZ)$. For readability, we write $H_2(G)$ without mentioning the coefficient ring $\ZZ$.

The paper is organized as follows. In Section~\ref{sec: pre-nt}, we remind the general notations of graph theory and group theory; we recapitulate the basics of WLOG groups, 
the homological version of the Bogomolov multiplier for any group $G$, i.e., $\widetilde{B}_0(G)$, and the basics of Bestvina--Brady groups.
In Section~\ref{sec: main}, we first analyze the free crossed module, and $\widetilde{B}_0(G)$ for an arbitrary WLOG group. In Section~\ref{sec: bb wlog}, we give the WLOG group structure for an arbitrary finitely presented Bestvina--Brady group and study its Schur multriplier and the functor $\widetilde{B_0}$. Our result not only gives generators for the Schur multiplier of such a group, but also gives a new proof of the structure of its Schur multiplier.
Then in~\ref{sec: even Arin gp}, we illustrate the approach to the computations of $\widetilde{B}_0(A)$ and the $H_2(A)$ of an arbitrary Artin group $A$ through its presentation as a WLOG group. Finally, in Section~\ref{sec: examples}, we exhibit a couple of examples applying the results depicted in the previous sections.

Our main results are the following:

\begin{thm*}[Theorem~\ref{thm: multi. vs. log}]
Let $\Gamma$ be a finite WLOG and G be the corresponding WLOG group. Then $\widetilde{B}_0(G)$ is finitely generated and a finite set of generators is computable.
\end{thm*}

\begin{thm*}[Theorem~\ref{thm: bb grp LOG}]
  Let $H_\Gamma$ be a finitely presented Bestvina--Brady group. Then the WLOG group presentation of $H_\Gamma$ is algorithmically computable. 
\end{thm*}

As an application of the above theorems, we prove the following theorem for finitely presented Bestvina--Brady groups:

\begin{thm*}[Theorem~\ref{thm: Schur of BB}]
Let $H_\Gamma$ be a finitely presented Bestvina--Brady group. Then the explicit basis for $H_2(H_\Gamma)$ is computable, and $\widetilde{B}_0(H_\Gamma)$ is trivial.    
\end{thm*}

We further extend our application towards the class of Artin groups by giving their WLOG group structure. It is worth mentioning that the existing results about (co)homology of Artin groups all focus on particular types Artin groups, for which the $K(\pi, 1)$ conjecture has been proved. There are very few properties that can be said for (co)homology of all Artin groups (except for their first integral homology, which are simply abelianizations). 
In the subsection~\ref{sec: even Arin gp}, we compute the homological version of Bogomolov multipliers and the second homology group of \emph{all} Artin groups without assuming an affirmative solution of the $K(\pi, 1)$ conjecture.

In this context, our result is the following:

 

\begin{thm*}[Theorem~\ref{thm: multiplier vs Artin}]
    Let $A$ be an Artin group. Then the WLOG group presentation of $A$ is computable. Also, the structure of $H_2(A)$ and $\widetilde{B}_0(A)$ follow from that particular presentation. 
\end{thm*}

As corollaries we discuss the same for \emph{even} Artin groups and right-angled Artin groups.


\section{Preliminaries and Notation}\label{sec: pre-nt}
\subsection*{Graph theory} We recall the basic definitions and fix the notations of graph theory. A finite \emph{simplicial graph} is a finite graph that have no loops and multi-edges.
Given a graph $\Gamma$, we denote the set of its vertices and edges by $V\Gamma$ and $E\Gamma$, respectively. The \emph{order} of a graph is its number of vertices. We denote by $e_{i,j}=(a_i, a_j)$ an edge connecting the vertices $a_i$ and $a_j$ and we say that the vertices $a_i$ and $a_j$ are \emph{adjacent}. Sometimes, we consider oriented edges and in this case, we denote by $o(e)$ and $t(e)$, the initial and the terminal vertices of the edge $e$, i.e. if $e=(a_i, a_j)$ is an oriented edge, then $e=(o(e), t(e))$. For any two vertices $a_i\neq a_j$ in $V\Gamma$, a path $p_{i,j}$ from $a_i$ to $a_j$ in $\Gamma$ is a sequence of vertices $a_i=a_{i_1}, a_{i_2}, \dots, a_{i_{d+1}}=a_j$ such that $a_{i_k}\in V\Gamma$ and $(a_{i_{k-1}}, a_{i_k})\in E\Gamma$, for $k=2, \dots, d+1$. A \emph{spanning tree} of $\Gamma$ is a subgraph of $\Gamma$ which is a tree and contains every vertex of $\Gamma$. A path from the vertex $a_i$ to the vertex $a_j$ along the spanning tree $T$ of $\Gamma$ is denoted by $[a_i,a_j]_T$.

Given any subset $V'$ of $V\Gamma$, the \emph{induced subgraph} (or full subgraph) on $V'$ is a graph $\Gamma'$ whose vertex set is $V'$, and two vertices are adjacent in $\Gamma'$ if and only if they are adjacent in $\Gamma$. 

The \emph{flag complex} defined by $\Gamma$ and denoted by $\triangle_\Gamma$, is a finite simplicial complex whose $1$ skeleton is $\Gamma$ and $\triangle_\Gamma$ contains an ($n$)-simplex bounding each complete induced subgraph (with $(n+1)$-vertices) of its 1-skeleton $\Gamma$. 
Given a graph $\Gamma$ we construct the flag complex $\triangle_\Gamma$ as follows: 
the vertex set is the ground vertex set $V\Gamma$ and a subset of cardinality $n+1$ is a $n$-simplex if and only if the induced subgraph is a $n$-clique. In the literature, the term clique complex is also used for the flag complex.

\subsection*{Group theory}\label{sec: pre gp th} We fix some notations used throughout the paper. Given a group $G$, and elements $g,h \in G$, we denote  by $g^h=h^{-1}gh$ the conjugate of $g$ by $h$, and we extend this notation to subsets $R,S\subseteq G$ writing $R^S=\{r^s \mid r\in R,\,\, s\in S\}$. The normal closure of $R\subseteq G$ in $G$ is $\left\langle \! \left\langle{R}\right\rangle \! \right\rangle_G =\langle{R^G} \rangle$, namely the smallest normal subgroup of $G$ containing $R$. The
commutator $[g, h]$ of elements $g$ and $h$ is defined by $[g, h] = g^{-1}h^{-1}gh$.
If $H$ and $K$ are subgroups of $G$, then we define $[H, K] = \langle \, [h, k] \, | \, h \in H, k \in K \, \rangle$. The \emph{commutator subgroup} $\gamma_2(G)$ of $G$ is defined to be the group $[G, G]$. The set
$\{\,[g, h] \,|\, g, h \in G \,\}$ of \emph{all} commutators of $G$ is denoted by $K(G)$. Let $H, K \trianglelefteq G$, the group $H \wedge K$ is generated by $h \wedge k$, where $h \in H$, $k \in K$ and satisfying the following relations,
\begin{itemize}
    \item[(i)] $hh' \wedge k = (h'^{(h^{-1})} \wedge k^{(h^{-1})})(h \wedge k)$,
    \item[(ii)] $h \wedge kk' = (h \wedge k)(h^{(k^{-1})} \wedge k'^{(k^{-1})})$,
    \item[(iii)] $x \wedge x =1$,
\end{itemize}
for all $h, h' \in H$, $k, k' \in K$ and for all $x \in H \cap K$.

\subsection{Word Labelled Oriented Graph groups}\label{sec: LOG}

 We start this subsection by defining word labelled oriented graph (WLOG) and WLOG groups (see~\cite{Harlander-Rosebrock}, \cite{Howie} for more details).

\begin{defn}\label{def: log groups}
    A \emph{word labelled oriented graph (WLOG)} is a graph
$\Gamma$ with vertex set $V = V\Gamma$, edge set $E = E\Gamma$, and initial and terminal vertex
maps $o, t : E \rightarrow V$ ; together with a third map $\lambda : E \rightarrow V^{*}$ called the \emph{labelling}, where $V^{*}$ denotes the set of non-empty words on the alphabet $V^{\pm 1}$ and each oriented edge is labeled by a non-empty word in $V^{\pm}$.
In the case where the underlying graph is a tree we speak of a word labelled oriented tree (WLOT).
To any WLOG $\Gamma$ we associate a presentation

\[
\mathcal{P} = \mathcal{P}(\Gamma) = \langle V\Gamma \, \, | \, \, o(e)\lambda(e) = \lambda(e) t (e) \rangle 
\]

of a group $G = G(\Gamma)$. 

Any group arising in this way is called a \emph{WLOG group} (a \emph{WLOT} group if $\Gamma$ is a tree).

\end{defn}

Due to the Theorem~\ref{thm: Ratcliffe} of Ratcliffe~\cite{Ratcliffe}, there is an interplay between WLOG groups and free crossed modules. We use this to prove Theorem~\ref{thm: multi. vs. log} (see below) and it was also used in \cite[Thm.~2.2]{N. D. Gilbert}.
Before, stating the Theorem~\ref{thm: Ratcliffe}, we remind the notion of \emph{weight}.
The \emph{weight} of a finitely generated group is the minimum number of elements needed to generate it as a normal subgroup. If $G$ is a group of weight $k$, then
a weight set in $G$ is a subset $X \subseteq G$ with $|X| = k$ whose normal closure is $G$.

\begin{thm}[J. G. Ratcliffe; \cite{Ratcliffe}]\label{thm: Ratcliffe}
Let $\varphi : C \rightarrow G$ be a crossed module, with $N = \im \varphi$ and $Q = G/N$. Then $C$ is a free crossed module on the subset $X \subseteq C$ if and only if
\begin{itemize}
    \item[(i)] the abelianisation $C^{\text{ab}}$ is a free $Q$-module, with the image of $X$ in $C^{\text{ab}}$ as a basis,
    \item[(ii)] the image of $X$ in $G$ is a weight set for $N$, 
    \item[(iii)] the map $\varphi_* : H_2(C) \rightarrow H_2(N)$ induced by $\varphi$ is trivial. 
\end{itemize}
\end{thm}

This theorem plays a central role for \cite[Thm.~2.2]{N. D. Gilbert} on the computation of the generators of the second Homology group $H_2(G)$, when $G$ is a WLOG group.

\subsection{Bogomolov multiplier and the functor $\widetilde{B}_0$}\label{sec: bog mul} 
Consider the group $G \wedge G$, called the \emph{non-abelian exterior square} of $G$. By construction (see~\ref{sec: pre gp th}), the map $\zeta \colon G \wedge G \rightarrow \gamma_2(G)$, given by $g \wedge g' \mapsto [g, g']$, $g, g' \in G$, is a well defined homomorphism of groups. Clearly $M(G) = \ker \zeta$ is central in $G \wedge G$. In~\cite{Miller}, Miller proved that there is a natural isomorphism between $M(G) \text{ and } H_2(G, \ZZ)$. As a direct consequence, $G \wedge G$ is naturally isomorphic to $\gamma_2(F )/[R, F ]$, where the group $G$ is given by a free presentation $G \cong F/R$. The \emph{Bogomolov multiplier} of a group $G$ is defined as the subgroup of the Schur multiplier consisting of the cohomology classes vanishing after restriction to all abelian subgroups of $G$.

\begin{equation*}
   B_0(G) = \ker[H^2(G, \QQ/\ZZ) \rightarrow \underset{A \leqslant G}{\bigoplus} H^2(A, \QQ/\ZZ)] 
\end{equation*}

$\res^G_A \colon H^2(G, \QQ/\ZZ) \rightarrow H^2(A, \QQ/\ZZ)$ is the usual cohomological restriction map. Let $H \leqslant G$. Then there is a corestriction
map $\cors^H_G \colon H_2(H, \ZZ) \rightarrow H_2(G, \ZZ)$. Let us define $M_0(G) = \langle \cors^A_G M(A) \,\mid \, A \leqslant G; A \text{ abelian} \rangle$.
This group can be described as a subgroup of $G \wedge G$ in the following way. Let $G$ be a group. Then $M_0(G) = \langle x \wedge y \, \mid \, x, y \in G, [x, y] = 1 \rangle$.
For a group $G$, the functor $\widetilde{B}_0$ is the following,
$$\widetilde{B}_0(G) = M(G)/ M_0(G),$$ 
where $M_0(G)$ is the subgroup of $M(G)$ generated by all $x \wedge y$ such that $x, y \in G$ commute. In the finite case, $\widetilde{B}_0(G)$ is thus (non-canonically) isomorphic to $B_0(G)$. The functor $\widetilde{B}_0(G)$ can be studied within the category of all
groups. In~\cite{Moravec}, P. Moravec proved a Hopf-type formula for $\widetilde{B}_0(G)$ by showing that if $G$ is given by a free presentation $G = F/R$, then

\begin{equation}\label{equ: Hopf type}
    \widetilde{B}_0(G) \cong \gamma_2(F ) \cap R/\langle K(F) \cap R\rangle,
\end{equation}

where $K(F)$ denotes the set of commutators in $F$. A special case of this was
implicitly used before by Bogomolov~\cite{Bogomolov}, and Bogomolov, Maciel and Petrov~\cite{Bogomolov-Maciel-Petrov}.
P. Moravec~\cite{Moravec} derived a five term exact sequence

\begin{equation}\label{equ: five term}
    \widetilde{B}_0(G) \rightarrow \widetilde{B}_0(G/N) \rightarrow N/\langle K(G) \cap N\rangle \rightarrow G^{\text{ab}} \rightarrow  {G/N}^{\text{ab}}\rightarrow 0,
\end{equation}

where $G$ is any group and $N$ a normal subgroup of $G$. This is a direct analogue of the
well-known five term homological sequence.




\subsection{Bestvina--Brady groups}\label{sec: fav triangles}

\begin{defn}\label{dfn: RAAG, BB}
Let $\Gamma$ be a finite simplicial graph with the vertex set $V\Gamma$ and the edge set $E\Gamma$. 
The \emph{right-angled Artin group} $G_{\Gamma}$ associated to $\Gamma$ has the following finite presentation:

\[
G_\Gamma = \bigl\langle V\Gamma \mid [a_i,a_j]=1 \text{ for each edge } (a_i,a_j) \in E\Gamma \bigr\rangle.
\]

Let $\varphi \colon G_\Gamma \rightarrow \ZZ$ be the group homomorphism sending all generators of $G_\Gamma$ to the generator $1$ of $\ZZ$. We call this epimorphism \emph{the canonical epimorphism} (from $G_\Gamma$ to $\ZZ$). The \emph{Bestvina--Brady group} $H_\Gamma$ associated with $\Gamma$ is the kernel $\ker \varphi$ of the canonical epimorphism $\varphi$. 
\end{defn}
In their influential work \cite{BB}, Bestvina and Brady described the homological finiteness properties of the kernels $H_\Gamma$ of the epimorphisms from $G_\Gamma$ to $\mathbb Z$ that map every standard generator of $G_\Gamma$ to $1$ in terms of the topology of the flag complex $\triangle_\Gamma$. More precisely, they prove that $H_\Gamma$ is finitely presented if and only if $\triangle_\Gamma$ is simply connected and $H_\Gamma$ is of type $FP_{n+1}$ if and only if $\triangle_\Gamma$ is $n$-acyclic.

In \cite{DL}, Dicks and Leary gave an explicit presentation for $H_\Gamma$. The generators in the presentation correspond to the edges $e_i$ of $\triangle_\Gamma$ and in the case when $\triangle_\Gamma$ is simply connected, it is shown that the relations are of the form $e_1^{\epsilon}e_2^{\epsilon}e_3^{\epsilon}$ for each directed $3$-cycles $(e_1, e_2, e_3)$ of $\triangle_\Gamma$ and $\epsilon = \pm 1$. This result gave an independent and purely algebraic proof that $H_\Gamma$ is finitely presented when $\triangle_\Gamma$ is simply connected.

\begin{thm}[\cite{DL}, Theorem 1]
Let $\triangle_\Gamma$ be connected. The group $H_\Gamma$ has a presentation with generators the set of directed edges of $\Gamma$, and relators all words of the form $e^n_1 e^n_2 \cdots e^n_{\ell}$, where $\ell,n \in \ZZ, n \geqslant 0,\ell \geqslant 2$, and $(e_1, \ldots,e_\ell)$ is a directed cycle in $\Gamma$.  
 
Furthermore, if the flag complex $\triangle_\Gamma$ is simply connected. Then $H_\Gamma$ has the following finite presentation

\begin{equation*}
H_\Gamma = \langle e \in E\Gamma \mid ef = fe, ef = g \text{ if }\triangle(e, f, g) \text{ is a directed triangle } \rangle,
\end{equation*}

where the inclusion $f \colon H_\Gamma \hookrightarrow A_\Gamma$ is given by $f({e}) = a_ia^{-1}_j$ for every edge ${e} = (a_i, a_j)$ of $\Gamma$.
\end{thm}

\begin{figure}[h]
    \centering
    \begin{tikzpicture}
    \tikzset{
    edge/.style={draw=black,postaction={on each segment={mid arrow=black}}}
} 
\node[fill=black!100, state, scale=0.10, vrtx/.style args = {#1/#2}{label=#1:#2}] (A) [vrtx=below/$a_i$]     at (0, 0) {};
\node[fill=black!100, state, scale=0.10, vrtx/.style args = {#1/#2}{label=#1:#2}] (C) [vrtx=above/$a_j$]    at (1, 1) {};
\node[fill=black!100, state, scale=0.10, vrtx/.style args = {#1/#2}{label=#1:#2}] (B) [vrtx=below/$a_k$]     at (2.5, 0) {};

\draw[edge] (A) -- (B) node[midway, below] {$g$};
\draw[edge] (C) -- (B) node[midway, above right] {$f$};
\draw[edge] (A) -- (C) node[midway, above left] {$e$};
    \end{tikzpicture}
    \caption{A directed triangle. 
    }
    \label{directed triangle}
\end{figure}
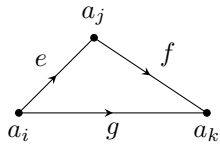

The Dicks-Leary presentation is not necessarily a minimal presentation, i.e., there are some redundant generators. Dicks-Leary considered all the edges of $\Gamma$ as the generators.
The simpler presentation was given by Papadima-Suciu in ~\cite{PS}. The authors proved that for the generators of $H_\Gamma$ it is enough to consider the edges of a spanning tree of $\Gamma$.

\begin{thm}[{\cite[Corollary 2.3]{PS}}]\label{thm: PS presentation}
If $\triangle_\Gamma$ is simply-connected, then $H_\Gamma$ has a presentation $H_\Gamma = F/R$, where $F$ is the free group generated by the edges of a spanning tree of $\Gamma$, and $R$ is a finitely generated normal subgroup of the commutator group $[F, F]$.
\end{thm}

Here is a couple of first examples of Bestvina--Brady groups. 

\begin{exam}
If $\Gamma$ is a complete graph on $n$ vertices, then any spanning tree has $n-1$ edges. Moreover, any two edges in the spanning tree form the two sides of a triangle in $\Gamma$. 
The corresponding Bestvina--Brady group has $n-1$ generators and any two of them commute, hence it is $\ZZ^{n-1}$. 
\end{exam}

\begin{exam}\label{free gp BB}
Now let $\Gamma$ be a tree on $n$ vertices.
The spanning tree is the graph $\Gamma$ itself and there are no triangles. 
The corresponding Bestvina--Brady group has $n-1$ generators and none of them commute, hence it is $F_{n-1}$ the free group on $n-1$ generators.
\end{exam}

\section{Main results}\label{sec: main}

Given a presentation $G = \langle X \mid R \rangle$ of a group there is a standard way of constructing a cell complex with fundamental group $G$, and the cell complex is called the \emph{presentation $2$-complex}. This has one vertex, an edge for each generator $x \in X$ and a $2$-cell for each relator $r\in R$. We record the following remark from~\cite{N. D. Gilbert, Harlander-Rosebrock} regarding the integral homology of the presentation $2$-complex of a WLOG group $G$. 

\begin{rem}\label{rem: classifying space}
  We fix the WLOG $\Gamma$. Let $K(\Gamma)$ be the presentation $2$-complex $K(\Gamma)$ of the WLOG group $G(\Gamma)$. Note that the each relator $o(e)\lambda(e) = \lambda(e)t(e)$ of the WLOG group $G(\Gamma)$, i.e., each relator of $\mathcal{P}(\Gamma)$, can be replaced by $o(e) = t(e)$ without altering the homology type of $K(\Gamma)$. The resulting 2-complex is the suspension of $(\Gamma \cup \{\text{point}\})$. Hence, $K(\Gamma)$ has the same integral homology as the suspension of $(\Gamma \cup \{\text{point}\})$.  
\end{rem}

In~\cite{Harlander-Rosebrock}, Harlander--Rosebrock discussed the asphericity of $K(\Gamma)$ and aspherical cyclically presented groups.

Let $G$ be an arbitrary group given by a free presentation $G= F/R$. For $g \in G$, let $C_G(g)$
denotes the centralizer of $g$ in $G$. Then using Moravec's formula to identify $\widetilde{B}_0(G)$ as $(\gamma_2(F) \cap R)/\langle K(F) \cap R \rangle$, a group homomorphism can be obtained,

\[
\begin{array}{cl}
     \Psi_g : & C_G(g) \rightarrow \widetilde{B}_0(G)\\
              & a  \mapsto [\tilde{a}, \tilde{g}] \, \langle K(F) \cap R \rangle,
\end{array}
\]

where $\tilde{a}, \tilde{g}$ are the preimages in $F$ of the elements $a$ and $g$ of $G$.

For any subset $X \subseteq G$, one obtains a homomorphism

\[
\begin{array}{cl}
   \Psi_X :  & *_{x \in X} \, C_G(x) \, \, \rightarrow \, \, \widetilde{B}_0(G), \\
             & (a_x)_{x\in X} \, \, \mapsto \, \, \Pi_{x\in X} \, (\Psi_X \, (a_x) \, ).
\end{array}
\]

In a similar manner, one can obtain a group homomorphism
$\psi_g \colon C_G(g) \rightarrow H_2(G)$ defined by $\psi_g(a) = [\tilde{a}, \tilde{g}][R, F]$ and $\psi_X \colon *_{x \in X} C_G(x) \rightarrow H_2(G)$ as defined above.

Now we prove one of our main results on the computation of the generators of $\widetilde{B_0}(G)$, where $G$ is a WLOG group. Our result is analogous to \cite[Thm.~2.2]{N. D. Gilbert} regarding $H_2(G)$.


\begin{thm}\label{thm: multi. vs. log}

\begin{enumerate}

\item [(a)] Let $G$ be a group of weight $m$ whose abelianisation $G^{\text{ab}}$ is a free abelian of rank $m$, and let $M$ be a weight set in $G$. Let $\varphi : C \rightarrow G$ be
the free crossed G-module on the inclusion $\iota: M \hookrightarrow G$. Then $C$ is a WLOG group and the kernel of $\varphi$, $\ker \varphi$ is isomorphic to $\widetilde{B}_0(G)$.

\item [(b)] Suppose $\Gamma$ is a finite WLOG with $n$ connected components and that $G = G(\Gamma)$. Let $S = \{v_1, \ldots v_n \}$ be a choice of vertices, one from each component of $\Gamma$. Then
the homomorphism $\Psi_S$ is surjective. Moreover, $\widetilde{B}_0(G)$ is generated by the $\Psi_S$ images of the elements $\lambda_i(g)$, where $g$ runs through a basis of the free group $*^n_{i=1} \pi_1(\Gamma, v_i)$.
\end{enumerate}
\end{thm}

\begin{proof} (a) Let $F$ be the free group on $(M \times G)$. Then $C$ is isomorphic to the group $F/\left\langle \! \left\langle R \right\rangle \! \right\rangle_F$,

\[
R = \{\, (m, g)^{-1}\, (m', g')\, (m, g)\, (m', g'g^{-1}\,\iota(x)\,g)^{-1}\, \, | \, \, m, m' \, \in M, \, g, g' \, \in G \}.
\]

Let $\Gamma$ be the WLOG with $V\Gamma = M \times G$; and an edge from $(m', g')$ to $(m', g'g^{-1}\,\iota(x)\,g)$ labelled by $(m, g)$ for all $m, m' \, \in M$ and for all $g, g' \, \in G$. Then it is straightforward that $C$ acquires a WLOG group structure with $\Gamma$ as the underlying WLOG. 

From the definition of the weight set, $\left\langle \! \left\langle M \right\rangle \! \right\rangle_G =G$. Since, $\varphi : C \rightarrow G$ be the free crossed G-module on the inclusion $\iota : M \hookrightarrow G$. Applying Theorem~\ref{thm: Ratcliffe}, $N = \im \varphi =G$. Thus $\varphi$ is surjective. Also, from the same Theorem~\ref{thm: Ratcliffe}, $\varphi_* : H_2(C) \rightarrow H_2(G)$ induced from $\varphi$ is trivial and $\varphi$ induces an isomorphism between the abelianizations  of $C$ and $G$.
Now we apply Moravec's five term exact sequence regarding Bogomolov multiplier for the extension $\ker \varphi \rightarrow C \rightarrow G$. Thus we get the following:

\[
\widetilde{B}_0(C) \xrightarrow {\,\widetilde{\varphi_*}\,} \widetilde{B}_0(G) \xrightarrow {\,\mu\,}\ker\varphi/\langle K(C) \cap \ker\varphi \rangle \rightarrow H_1(C) \xrightarrow {\,\nu\,} H_1(G) \rightarrow 0
\]

Since the map $\varphi_* : H_2(C) \rightarrow H_2(G)$ is trivial, the map $\widetilde{\varphi_*} \colon \widetilde{B}_0(C) \rightarrow \widetilde{B}_0(G)$ is also trivial. Hence we have $\ker \mu$ is trivial. Also note that $\langle K(C)\, \cap \, \ker\varphi \rangle$ is trivial as $\ker \varphi$ is central.
The only remaining part is to prove that $\mu$ is onto and it follows directly from the fact that $\nu$ is an isomorphism. Thus we have the announced isomorphism.

(b) First we choose a spanning tree $T_i$ from each connected component of $\Gamma$ for $i = 1, 2, \ldots, n$. The union of the chosen spanning trees is then a spanning forest $\mathcal{T}$ in $\Gamma$. Let $\mathcal{G}$ be the WLOG group defined on the spanning forest $\mathcal{T}$. Then $G= \mathcal{G}/ N$, where $N$ is the normal closure of the set $R_\mathcal{T}$ in $\mathcal{G}$,

\begin{equation}\label{equ: for N}
   R_\mathcal{T} = \{ o(e) \, \lambda(e) \, t(e)^{-1} \, \lambda(e)^{-1} \, \, | \, \, e \in E\Gamma\setminus E\mathcal{T}\}.
\end{equation}

Now we consider the fundamental group of each connected component of $\Gamma$. In other words, $\pi_1(\Gamma, v_i)$ for each $v_i \in S$. It is obvious that the labelling function $\lambda \colon E\Gamma \rightarrow V^{*}$ induces a group homomorphism $\lambda_i \colon \pi_1(\Gamma, v_i) \rightarrow \mathcal{G}$. We claim that $[\lambda_i(g_i), v_i] \in N$ for all $g_i \in \pi_1(\Gamma, v_i)$. Let $g_i$ be an arbitrary element of $\pi_1(\Gamma, v_i)$. Then, $\lambda_i(g_i) = g_1 \lambda_i(e) {g_2}^{-1}$, where $e \in E\Gamma \setminus ET_i$, $g_1$ is the $\lambda_i$-image of the edge path $[v_i, o(e)]_{T_i}$ along the spanning tree $T_i$ and $g_2$ is the $\lambda_i$-image of the edge path $[v_i, t(e)]_{T_i}$ along the spanning tree $T_i$. Now applying the relations of the presentation of $\mathcal{P}(\Gamma)$ for the edges $e \in ET_i$, we have $[\lambda_i(g_i), v_i] = \lambda_i(g_i)^{-1} v_i^{-1} \lambda_i(g_i) v_i = g_2 \lambda_i(e)^{-1} o(e)^{-1} \lambda_i(e) t(e) g_2^{-1} \in N$. Thus we establish our claim and which also implies that
$\lambda_i(g_i)N \in C_G(v_i)$. Let us consider $\{ x_{i,1}, \ldots, x_{i,m_i}\}$ as the free basis for the free group $\pi_1(\Gamma, v_i) \cong F_{m_i}$ for each $i = 1, 2, \ldots n$. Also we note that as $[\lambda_i(g_i), v_i] \in N$, modulo the relations of $\mathcal{G}$, we may replace the set $R_{\mathcal{T}}$ in~\eqref{equ: for N} by the following set,

\[
X_\mathcal{T} = \bigcup_{i=1}^{n} \{ [\lambda_i(x_{ij}), v_i], 1 \leq j \leq m_i \}.
\]

Thus we have $N = \left\langle \! \left\langle X_\mathcal{T} \right\rangle \! \right\rangle_\mathcal{G}$. Let us call $\langle K(\mathcal{G}) \cap N \rangle = H$. Now we consider the following two homomorphisms $\mu_i \colon  \pi_1(\Gamma, v_i) \rightarrow C_G(v_i)$ and $\nu_i \colon  \pi_1(\Gamma, v_i) \rightarrow N/H$ induced by the homomorphism $\lambda_i$,

\[
\begin{array}{lr}
\begin{array}{cc}
   \mu_i \colon & \pi_1(\Gamma, v_i) \rightarrow C_G(v_i)  \\
     & g_i \mapsto \lambda(g_i)N
\end{array}    
&  
\begin{array}{cc}
  \nu_i \colon   &  \pi_1(\Gamma, v_i) \rightarrow N/H\\
     & g_i \mapsto [\lambda_i(g_i), v_i] H.
\end{array}
\end{array}
\]

Then, $\mu_i$ and $\nu_i$ extend to $\mu \colon  *^n_{i=1}\pi_1(\Gamma, v_i) \rightarrow *^n_{i=1}C_G(v_i)$ and $\nu \colon  *^n_{i=1}\pi_1(\Gamma, v_i) \rightarrow N/H$ respectively. As, $N = \left\langle \! \left\langle X_\mathcal{T} \right\rangle \! \right\rangle_\mathcal{G}$, it is straightforward that $\nu$ is surjective.
Applying Moravec's five-term exact sequence 

\[
\widetilde{B}_0(\mathcal{G}) {\rightarrow} \widetilde{B}_0(G) \xrightarrow{\, f\, } N/H \rightarrow H_1(\mathcal{G}) {\rightarrow} H_1(G) \rightarrow 0,
\]

of the group extension $N \rightarrow \mathcal{G} \rightarrow G$, we will prove that $\widetilde{B}_0(G) \cong N/H$. Note that $H_1(\mathcal{G}) \cong H_1(G)$. This follows from the fact that the relation $o(e)\lambda(e)=\lambda(e)t(e)$ becomes $o(e)=t(e)$ in $\mathcal{G}^{\text{ab}}$. Thus all the vertices of one connected component get identified and so $\mathcal{G}^{\text{ab}} \cong \ZZ^n$, $n$ is the number of connected components in $\Gamma$. Similarly, we have $G^{\text{ab}}\cong \ZZ^n$. Following the same argument and Remark~\ref{rem: classifying space} we can assume the resulting complex as the bouquet of $n$ circles labelled by $v_i$'s where each $v_i$ is representing each connected component of $\Gamma$ without alterning the homology type of presentation $2$-complex $K(\mathcal{T})$ of $\mathcal{G}$. Hence, $H_2(\mathcal{G})$ is a homomorphic image of $H_2(K(\mathcal{T})) =0$ and $\widetilde{B}_0(\mathcal{G}) = M(\mathcal{G})/M_0(\mathcal{G}) \cong H_2(\mathcal{G})/M_0(\mathcal{G})$ (see Section~\ref{sec: bog mul}). So, we have $\ker f$ is trivial, i.e., $f$ is injective. On the other hand, as, $H_1(\mathcal{G}) \cong H_1(G)$, from the five-term exact sequence we get $f$ is onto as well which in turn implies that $\widetilde{B}_0(G) \cong N/\langle K(\mathcal{G}) \cap N \rangle$. 
We have already showed $\nu$ is surjective, and $\nu$ factors through $\Psi_S$. Hence, $\Psi_S$ is also surjective and in particular $\widetilde{B_0}(G)$ is generated by the $\Psi_S$ images of the elements $\lambda_i(g)$, where $g$ runs through a basis of the free group $*^n_{i=1} \pi_1(\Gamma, v_i)$.
This completes the proof.
\end{proof}





\subsection{Bestvina--Brady group as WLOG groups.}\label{sec: bb wlog}

In~\cite{N. D. Gilbert}, N. D. Gilbert gave the WLOG presentation for the right-angled Artin groups and the Braid groups. And then by applying Theorem~\ref{thm: Ratcliffe} N. D. Gilbert gave an explicit description of the Schur multiplier of the group $G$, where $G$ could be a RAAG or a Braid group (see~\cite[Thm.~2.2(b)]{N. D. Gilbert}).
This result serves as a motivation to explore the same in the family of Bestvina--Brady groups and Artin groups. In the present article, we also extend the result for the computation of the Bogomolov multiplier for the aforementioned families of groups, which also include right-angled Artin groups.

\begin{defn}\label{complement}
Let $\Gamma$ be a connected graph and $\triangle$ be a triangle of $\Gamma$. Let $\Gamma'$ be the graph with $V\Gamma'=V\Gamma$ and $E\Gamma'=E\Gamma\setminus E\triangle$. Let $S$ be the isolated vertices of $\Gamma'$ and $V^c=V\Gamma \setminus S$. The \emph{edge-set ~complement} of $\triangle$ is the induced subgraph of $\Gamma$ generated by the set of vertices $V^c$. We will denote the edge-set complement of $\triangle$ in $\Gamma$ by $\triangle_\Gamma^c$.

A triangle $\triangle$ of $\Gamma$ is said to be an \emph{internal triangle} if its intersection with the edge-set complement $\triangle_\Gamma^c$ is neither one vertex nor one edge.
A triangle $\triangle$ of $\Gamma$ is said to be a \emph{strictly internal triangle} if it is contained in an induced $K_n$, $n \geq 4$.
\end{defn}
In Figure~\ref{fig:graph with int. tri.} the triangles $\triangle(a_2, a_4, a_5)$ and $\triangle(a_2, a_3, a_4)$ are the internal triangles, while $\triangle(a_3, a_4, a_7)$ is \emph{not} an internal triangle.
\begin{figure}[H]
    $$
    \begin{array}{ccc}
\begin{tikzpicture}[shorten >=1pt,node distance=17.5cm,auto]
\tikzset{
    edge/.style={draw=black,postaction={on each segment={mid arrow=black}}}
}
\node[fill=black!100, state, scale=0.10, vrtx/.style args = {#1/#2}{label=#1:#2}] (1) [vrtx=left/$a_6$] {};

\node[fill=black!100, state, scale=0.10, vrtx/.style args = {#1/#2}{label=#1:#2}] (2) [vrtx=above/$a_4$] [ below right of = 1] {};

\node[fill=black!100, state, scale=0.10, vrtx/.style args = {#1/#2}{label=#1:#2}] (7) [vrtx=right/$a_7$] [right of = 2] {};

\node[fill=black!100, state, scale=0.10, vrtx/.style args = {#1/#2}{label=#1:#2}] (3) [vrtx=left/$a_5$] [ below left of = 1] {};

\node[fill=black!100, state, scale=0.10, vrtx/.style args = {#1/#2}{label=#1:#2}] (4) [vrtx=right/$a_3$] [ below right of = 2] {};

\node[fill=black!100, state, scale=0.10, vrtx/.style args = {#1/#2}{label=#1:#2}] (5) [vrtx=below/$a_2$] [ below left of = 2] {};

\node[fill=black!100, state, scale=0.10, vrtx/.style args = {#1/#2}{label=#1:#2}] (6) [vrtx=left/$a_1$] [ below left of = 3] {};
\draw[edge] (5) -- (6) node[midway, below] {};
\draw[edge] (5) -- (3);
\draw[edge] (5) -- (2);
\draw[edge] (5) -- (4);
\draw[edge] (1) -- (2) node[midway, right] {};
\draw[edge] (6) -- (3) node[midway, left] {};
\draw[edge] (3) -- (2);
\draw[edge] (3) -- (1) node[midway, left] {};
\draw[edge] (2) -- (4);
\draw[edge] (2) -- (7) node[midway, above] {};
\draw[edge] (7) -- (4) node[midway, right] {};
\end{tikzpicture}
  \end{array}
  $$
\caption{A graph with internal triangles.}
\label{fig:graph with int. tri.}
\end{figure}
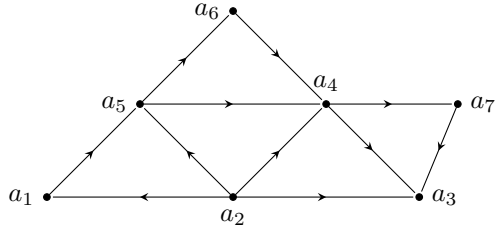

It is straightforward to note that all strictly internal triangles are internal.

\begin{defn}
Let $T$ be a spanning tree of $\Gamma$. 
A triangle $\triangle$ of $\Gamma$ is a \emph{favourable triangle}  with respect to $T$ if, either it has exactly $2$ edges in $ET$ or it is strictly internal.  
Otherwise, we say that $\triangle$ is \emph{unfavourable}.
\end{defn}

\begin{nt}
A graph $\Gamma$ can have several spanning trees. 
So, it may very well happen that a triangle $\triangle$ is \emph{favourable} in one spanning tree $T_1$ of $\Gamma$ but not favourable in another spanning tree $T_2$ of $\Gamma$. 
Hence to discuss favourable triangles we have to fix a spanning tree $T$ of $\Gamma$. 
We choose a spanning tree  with the \emph{maximal} number of \emph{favourable} triangles or equivalently, the number of \emph{unfavourable} triangles is \emph{minimal}.
\end{nt}

\begin{defn}
    A spanning tree $T$ of $\Gamma$ is said to be a favourable spanning tree if $\Gamma$ has the minimal number of unfavourable triangles with respect to $T$.
\end{defn}

\begin{thm}\label{thm: bb grp LOG}
  Let $H_\Gamma$ be a finitely presented Bestvina--Brady group. Then the WLOG group presentation of $H_\Gamma$ is algorithmically computable. 
\end{thm}

\begin{proof}
We choose a spanning tree of $\Gamma$ with the number of unfavourable triangles is minimal. Thus let $T$ be the chosen favourable spanning tree of $\Gamma$ and $e$ be an arbitrary edge of $\Gamma$. Then there are two obvious cases--- case-1: $e \in ET$ and case-2: $e \notin ET$ and each case has the following three possibilities:

\begin{itemize}
    \item[(i)] $e$ is not an edge of any triangle of $\Gamma$;
    \item[(ii)] $e$ is an edge of a single triangle $\triangle_1$ of $\Gamma$, i.e., $e \in E\triangle_1$;
    \item[(iii)] $e$ is a common edge of two adjacent triangles $\triangle_1, \triangle_2$ of $\Gamma$, i.e., $e \in E\triangle_1 \cap E\triangle_2$.
\end{itemize}

Now we construct a WLOG $\widetilde{\Gamma}$ such that $G(\widetilde{\Gamma}) \cong H_\Gamma$. The set of vertices of $\widetilde{\Gamma}$, i.e., $V\widetilde{\Gamma}$ is the edges of the spanning tree $T$ and exactly one vertex, namely $w(e)$, for each unfavourable triangle $\triangle$ (w.r.t $T$) which has $0$ edge in $ET$ and $e \in E\triangle$. For any arbitrary edge $e \in E\Gamma \setminus ET$, note that $w(e)$ is the word expression  of $e$ with respect to the generating set $ET$ of $H_\Gamma$ (see Theorem~\ref{thm: PS presentation}).

If we are in~\emph{case 1-(i)}, $e$ would be an isolated vertex.

For~\emph{case 1-(ii)}, let $\triangle_1$ be a favourable triangle with respect to $T$. Then $\triangle_1$ has two edges in $ET$, say, $e \text{ and }f$. We put a loop at $e$ labelled by $f$. Otherwise, if $\triangle_1$ is not a favourable triangle, we express the second edge, $f$ of $\triangle_1$ as a word in $ET$, say $w(f)$. This is always possible due to Theorem~\ref{thm: PS presentation}. And, we put a loop at $e$ labelled by $w(f)$.

Finally, for~\emph{case 1-(iii)}, there are following cases:

\begin{itemize}
    \item[(a)] $\triangle_1$ has exactly one edge (i.e., $e$) in $T$.
    \item[(b)] $\triangle_1$ has two edges (say, $e$, $f$) in $T$.
    \item[(a')] $\triangle_2$ has exactly one edge (i.e., $e$) in $T$.
    \item[(b')] $\triangle_2$ has two edges (say, $e$, $f'$) in $T$.
\end{itemize}

For $(a), (a')$: let $f \in E\triangle_1$ and $f' \in E\triangle_2$ with $f, f' \notin ET$. As before, we express $f, f'$ as words on $ET$, say $w(f), \, w(f')$ respectively. We attach two loops at $e$, one is labelled by $w(f)$ and the other one is labelled by $w(f')$.
If  $(a), (b')$ occur together, we put two loops at $e$--- one is labelled by $w(f)$ and other one is labelled by $f'$. Similarly, when $(b), (a')$ occur together, we have two loops at $e$ and those are labelled by $f$ and $w(f')$ respectively. For $(b), (b')$: we put a loop at $f$ labelled by $e$ and a loop at $e$ labelled by $f'$.



Note that we ignore \emph{case 2-(i)}, as in the mentioned presentation of $H_\Gamma$ (see Theorem~\ref{thm: PS presentation}) neither $e$ is a generator nor $w(e)$ is in the expression of any relator. For \emph{case 2-(ii)}, without the loss of generality, we assume $\triangle_1$ has $0$ edge in $ET$ (otherwise, renaming $e$, it would be \emph{case 1- (ii)}). We consider the two edges $e, f$ of $\triangle_1$. As both the edges $e, f$ are outside of the spanning tree, we express both of them as words on $ET$. We put the vertex labelled by $w(e)$ and attach a loop at $w(e)$ labelled by $w(f)$. Finally, for \emph{case 2-(iii)} again without the loss of generality, we assume both $\triangle_1$ and $\triangle_2$ have $0$ edge in $ET$ (otherwise, it would be \emph{case 1- (iii)}). Let $e,f \in E\triangle_1$ and $e, f' \in E\triangle_2$. We attach two loops at $e$, one is labelled by $w(f)$ and the other one is labelled by $w(f')$.
\end{proof}



Now we describe the structure of $H_2(H_\Gamma)$ and $\widetilde{B}_0(H_\Gamma)$ for a finitely presented Bestvina--Brady group $H_\Gamma$. We apply a purely combinatorial way to compute these homological versions of group theoretical invariants for any finitely presented Bestvina--Brady group; moreover, these are obtained from the defining graph of its ambient right-angled Artin group.

\begin{thm}\label{thm: Schur of BB}
Let $H_\Gamma$ be a finitely presented Bestvina--Brady group. Also let $\widetilde{\Gamma}$ is the WLOG such that $G(\widetilde{\Gamma}) \cong H_{\Gamma}$. Then, $H_2(H_\Gamma)$ is the free abelian of rank $m$, where $m$ is the number of loops in $\widetilde{\Gamma}$. Moreover, the basis of $H_2(H_\Gamma)$ is computable. And, $\widetilde{B}_0(H_\Gamma)$ is trivial.
\end{thm}

\begin{proof} The graph $\Gamma$ is connected and $\triangle_\Gamma$ is simply connected as $H_\Gamma$ is finitely presented. Let $V\Gamma = \{ a_1, \ldots, a_{n+1}\}$ and we denote the edge joining the vertices $a_i$ and $a_j$ as $e_{i,j}$. We choose $T$ to be the favourable spanning tree of $\Gamma$ and so $|ET|=n$ and $\Gamma$ has minimal number of unfavourable triangles with respect to $T$; let $\Gamma$ has $\ell$ many unfavourable triangles with respect to the tree $T$. We rename the edges of $T$ as $v_i$ and let $S= \{ v_1, \ldots, v_n \}$.
    From Theorem~\ref{thm: PS presentation}, $H_{\Gamma} \cong F_n/R$, where $F_n (\cong F(S))$ is the free group generated by the edges of $T$. First, we compute $\widetilde{\Gamma}$ as described in Theorem~\ref{thm: bb grp LOG} and then consider the spanning forest, say $\mathcal{T}$, for $\widetilde{\Gamma}$. Hence, we have $V\widetilde{\Gamma} = S \sqcup L = \{ v_1, \ldots, v_n, w_1, \ldots, w_\ell\}$, where $w_i$ is a word representative of a single edge of each unfavourable triangle for $i= 1, \ldots, \ell$. From the construction it is clear that $\widetilde{\Gamma}$ has $n + \ell$ many connected components. Also, $\mathcal{T}$ is a collection of $n+ \ell$ isolated vertices namely $v_1, \ldots, v_n, w_1, \ldots, w_\ell$. Let $\mathcal{G} (= G(\mathcal{T}))$ is the WLOG group corresponding to $\mathcal{T}$. Note that $\mathcal{G} \cong F(S) \cong F_n$ as $w_1, \ldots, w_\ell$ are redundant generators. Because, $w_i = w(v_1, \ldots, v_n)$ for $i= 1, \ldots, \ell$. Now we consider the obvious quotient map $\mathcal{G} \rightarrow H_\Gamma$ and let $N$ be the kernel of the quotient map $\mathcal{G} \rightarrow H_\Gamma$. From~\cite[Thm.~2.2(b)]{N. D. Gilbert}, $H_2(H_\Gamma) \cong N/[N, F_n]$
   
    The inclusion-induced mapping $N/[N, F_n] \rightarrow [\,F_n, F_n\,]/ [\,[F_n, F_n], F_n\,]$ is injective. We are aware of $[\,F_n, F_n\,]/ [\,[F_n, F_n], F_n\,] \cong H_2(F_n^{\text{ab}})$ and $H_2(F_n^{\text{ab}})$ is free abelian of rank $n(n-1)/2$. We remind the reader from the proof of Theorem~\ref{thm: multi. vs. log} that the labelling function $\lambda \colon E\widetilde{\Gamma} \rightarrow S^*$ induces a group homomorphism $\lambda_i \colon \pi_1(\widetilde{\Gamma}, v_i) \rightarrow F_n$.
    From~\cite[Thm.~2.2(b)]{N. D. Gilbert}, $H_2(H_\Gamma)$ is generated by the images under $\psi_S$ of elements $\lambda_i(g)$, where $g$ runs through a basis of the free group $*^n_{i=1} \pi_1(\widetilde{\Gamma}, v_i)$. 
    Since $H_2(H_\Gamma) \cong N/[N, F_n]$, $H_2(H_\Gamma)$ is also a free abelian group of rank $m$, where $m$ is the number of loops in $\widetilde{\Gamma}$. In fact, $H_2(H_\Gamma)$ is generated by the images of the loop labels under $\psi_S$. Thus the basis of $H_2(H_\Gamma)$ is algorithmically computable through the map $\psi_S$. This completes the proof.

    From Theorem~\ref{thm: multi. vs. log}, $\widetilde{B_0}(H_\Gamma) \cong N/\langle K(F_n) \cap N \rangle$. Clearly, $\langle K(F_n) \cap N \rangle$ is generated by the commuting relations arise from the triangles of $\Gamma$ and which is same as $N$. This holds as $E\widetilde{\Gamma}$ consists of loops only. Also, note that each triangle of $\Gamma$ represents a copy of $\ZZ^2$ in $H_\Gamma$. Hence it follows that $\widetilde{B_0}(H_\Gamma)$ is trivial.
    \end{proof}


\subsection{Artin groups as WLOG groups.}\label{sec: even Arin gp} 
Our motive is to give WLOG group structure to \emph{all} Artin groups and then compute their Schur multipliers from~\cite[Thm.~2.2(b)]{N. D. Gilbert} and $\widetilde{B_0}$ from Theorem~\ref{thm: multi. vs. log} without studying their Salvetti complexes, which is conjectured to be a model for the corresponding classifying spacees of the Artin
groups. In other words, without assuming the $K(\pi, 1)$ conjecture and also, Howlett's result on the second integral homology of Coxeter groups.

\begin{thm}\label{thm: multiplier vs Artin}
    Let $A$ be an Artin group. Then the WLOG group presentation of $A$ is computable. Moreover, $H_2(A)$ and $\widetilde{B}_0(A)$ can be derived from that particular presentation. 
\end{thm}

\begin{proof}
    Let $\Gamma$ be the Artin-Tits system corresponding to the given Artin group $A$. We construct a WLOG $\widetilde{\Gamma}$ with $V\widetilde{\Gamma} =\{a_1, \ldots, a_n \}$. If $m_{ij}$ is \emph{even}, we put a loop at $a_i$ labelled by $a_j(a_ia_j)^{\frac{m_{ij}}{2}-1}$ and if $m_{ij}$ is \emph{odd}, we put an edge $e_{i,j}$ labelled by $(a_ja_i)^{\frac{m_{ij}-1}{2}}$. We remind from Section~\ref{sec: intro} that the relations are of the form $\langle a_i, a_j \rangle^{m_{ij}} = \langle a_j, a_i \rangle^{m_{ij}}$. And, 
    $\langle a_i, a_j \rangle^{m_{ij}} = (a_ia_j)^{m_{ij}/2}$ if $m_{ij}$ is even and $\langle a_i, a_j \rangle^{m_{ij}} = (a_ia_j)^{(m_{ij}-1)/2}a_i = a_i(a_ja_i)^{(m_{ij}-1)/2}$ if $m_{ij}$ is odd.
    Hence, a direct computation shows that $G(\widetilde{\Gamma}) \cong A$. Clearly, $\widetilde{\Gamma}$ is disconnected and having $\ell$ many connected components, where $\ell = 1+ \# \{ a_i \in V\Gamma \mid m_{ij} = \text{ even or } m_{ij} = \infty, \, \forall a_j \in V\Gamma \}$. We choose one vertex from each connected component. Without the loss of generality let $S = \{a_1, \ldots, a_\ell\}$ be a choice of vertices, one from each connected component of $\widetilde{\Gamma}$. Also, let $\mathcal{T}$ be the spanning forest of $\widetilde{\Gamma}$ and $G(\mathcal{T}) = \mathcal{G}$. Then $A \cong G(\widetilde{\Gamma}) = \mathcal{G}/ N$, where $N$ is the normal closure of the set $R_\mathcal{T}$ in $\mathcal{G}$,

\begin{equation*}
   R_\mathcal{T} = \{ o(e) \, \lambda(e) \, t(e)^{-1} \, \lambda(e)^{-1} \, \, | \, \, e \in E\widetilde{\Gamma}\setminus E\mathcal{T}\}.
\end{equation*}

From~\cite[Thm.~2.2(b)]{N. D. Gilbert}, $H_2(A) \cong N/[N, \mathcal{G}] \cong \ZZ^{m}$, where $m = |E\widetilde{\Gamma}\setminus E(\mathcal{T})|$ and is generated by the $\psi_S$-images of $R_\mathcal{T}$.
On the other hand, applying Theorem~\ref{thm: multi. vs. log}, $\widetilde{B}_0(A) \cong N/\langle K(\mathcal{G}) \cap N\rangle$. Note that $\widetilde{B}_0(A)$ is trivial if and only if $\widetilde{\Gamma}$ has only loops outside the spanning tree $\mathcal{T}$. Hence, $\widetilde{B}_0(A) \cong \ZZ^{p}$, where $p$ is the number of non-loop edges of $\widetilde{\Gamma}$ outside the spanning tree $\mathcal{T}$. Furthermore, the basis of $\widetilde{B}_0(A)$ is $\Psi_S$-images of the elements of $N$ that come from the non-loop edges outside the set $E\mathcal{T}$.
\end{proof}

As we discussed in Section~\ref{sec: intro}, an even Artin group $A_{\Gamma}$ has the following presentation:

\[
A_\Gamma = \langle a_1, \ldots, a_n \, | \, (a_ia_j)^{m_{ij}/2} \, = \, (a_ja_i)^{m_{ij}/2}; \,  i \neq j, \text{ and } m_{ij} \text{ is even} \rangle,
\]

since we omit the relation between $a_i$ and $a_j$, when $m_{ij} = \infty$. In fact each relator of $A_\Gamma$ can be replaced by the following,

\begin{equation}\label{eq: even pre}
   a_ia_j(a_ia_j)^{\frac{m_{ij}}{2} - 1} \, = \, a_j(a_ia_j)^{\frac{m_{ij}}{2} - 1}a_i; \,  i \neq j, \text{ and } m_{ij} \text{ is even }.
\end{equation}


\begin{cor}\label{cor: multiplier vs even Artin}
    Let $A_\Gamma$ be an even Artin group. Then the WLOG group presentation of $A_\Gamma$ is computable. Also, the structure of $H_2(A_\Gamma)$ and $\widetilde{B}_0(A_\Gamma)$ follow from that particular presentation. 
\end{cor}

\begin{proof}
First we emphasize on the construction of $\widetilde{\Gamma}$ with the aim that $G(\widetilde{\Gamma}) \cong A_\Gamma$.
We construct $\widetilde{\Gamma}$ in the following way. The vertices of $\widetilde{\Gamma}$ are $a_1, \ldots, a_n$---let us call $V = \{ a_1, \ldots, a_n\}$; and a loop at each vertex $a_i$ labelled by $a_j(a_ia_j)^{\frac{m_{ij}}{2}-1}$ if and only if  $i \neq j$ and $m_{ij}$ is even. It is straightforward from~\eqref{eq: even pre}, $G(\widetilde{\Gamma}) \cong A_{\Gamma}$. Note that $\widetilde{\Gamma}$ is disconnected and having as many connected components as the order of $\widetilde{\Gamma}$. Now, let $G(\widetilde{\Gamma})$ be the WLOG group associated to the WLOG $\widetilde{\Gamma}$. Clearly, $G(\widetilde{\Gamma}) \cong A_{\Gamma} \cong F(V)/R$, where $R$ is the normal closure of the set $\{ a_ia_j(a_ia_j)^{\frac{m_{ij}}{2} - 1} \, = \, a_j(a_ia_j)^{\frac{m_{ij}}{2} - 1}a_i; \,  i \neq j\}$ in $F(V)$. The spanning
forest in $\widetilde{\Gamma}$ is just $V$; the WLOG group based on the spanning forest of $\widetilde{\Gamma}$, $\mathcal{G}$ is the free group $F_n = F(V)$. 
Let us consider the canonical quotient map $f \colon F_n \longrightarrow A_{\Gamma}$ and $N$ be the kernel of the map $f$. From~\cite[Thm.~2.2(b)]{N. D. Gilbert} and following the same procedure as discussed in the proof of Theorem~\ref{thm: Schur of BB}, $H_2(A_\Gamma) \cong N/[N, F_n]$. Since, $H_2(A_\Gamma)$ is subgroup of $H_2(F_n^{\text{ab}})$, $H_2(A_\Gamma)$ is also a free-abelian group. In fact, the rank of the group $H_2(A_\Gamma)$ is the number of loops present in $\widetilde{\Gamma}$; and is equal to $|E\Gamma|$. We remind that $\Gamma$ is the Artin-Tits system associated to $A_\Gamma$. In short, $H_2(A_\Gamma) \cong \ZZ^r$, where $r$ is the number of edges in the graph $\Gamma$. Finally, we compute $\widetilde{B_0}(A_\Gamma)$. Note that, $\langle K(F_n) \cap N \rangle$ is generated by the loop labels of $\widetilde{\Gamma}$.
Hence, $\widetilde{B_0}(A_\Gamma) \cong N/\langle K(F_n) \cap N \rangle$ is trivial.
\end{proof}

This leads to the following corollary.

\begin{cor}\label{cor: raag}
Let $G_\Gamma$ be a right-angled Artin group. Then $\widetilde{B}_0(G_\Gamma)$ is trivial.   
\end{cor}

\begin{proof}
    Let $\Gamma$ be a finite simplicial graph with vertex set $V\Gamma = \{ a_1, a_2, \ldots, a_n\}$. Then $G_\Gamma = \bigl\langle V\Gamma \mid [a_i,a_j]=1 \text{ for each edge } (a_i,a_j) \in E\Gamma \bigr\rangle$. First we note that $G_\Gamma$ is a WLOG group. The associated WLOG $\widetilde{\Gamma}$ having the vertex set $\{a_1, a_2, \ldots, a_n \}$ and, at the vertex $a_i$ a loop labelled by $a_j$ whenever $(a_i, a_j) \in E\Gamma$, $(1 \leq i, j \leq n)$. In a similar fashion as in the proof of Theorem~\ref{thm: Schur of BB}, the spanning forest $\mathcal{T}$ of $\widetilde{\Gamma}$ is a collection of $n$ many vertices, namely $\{a_1, a_2, \ldots, a_n \}$. $G(\widetilde{\Gamma}) \cong G_\Gamma \cong F_n/N$, where $F_n$ is a finitely generated free group generated by $V\Gamma$ and $N$ is the normal closure of the commuting relations coming from the edges of $\Gamma$. Clearly, $\langle K(F) \cap N \rangle = N$ and hence from Theorem~\ref{thm: multi. vs. log}, $\widetilde{B}_0(G_\Gamma) \cong N/\langle K(F) \cap N \rangle$ is trivial.
\end{proof}


\subsection{Examples}~\label{sec: examples} We end the article by giving a couple of examples. The illustrations represent both Bestvina--Brady group and Artin group as WLOG groups and produce their respective Schur multipliers and the functor $\widetilde{B_0}$.

\begin{exam}
Let $\Gamma$ be the graph in \Cref{fig:exam 1}. 
Choosing the spanning tree $T = \{e_{1,2}, \, e_{2,3}, \, e_{2,4}, \, e_{2,5}, \, e_{4,6}\}$ as indicated, the presentation of the Bestvina--Brady group is given as follows:
\[
H_\Gamma = \langle e_{1,2}, \, e_{2,5}, \, e_{2,4}, \, e_{2,3}, \,   e_{4,6} \mid [e_{1,2}, e_{2,5}],\, [e_{2,5}, e_{2,4}],\, [e_{2,4}, e_{2,3}], \, [e_{4,6}, \, e_{2,5}^{-1} e_{2,4} ]\rangle
\]
This particular $H_\Gamma$ is not isomorphic to any right-angled Artin group (see \cite[Proposition 9.4]{PS} for details). 

\begin{figure}[H]
    \centering
\begin{tikzpicture}[shorten >=1pt,node distance=20cm,auto]
\tikzset{
    edge/.style={draw=black,postaction={on each segment={mid arrow=black}}}
}
\node[fill=black!100, state, scale=0.10, vrtx/.style args = {#1/#2}{label=#1:#2}] (1) [vrtx=left/$a_6$] {};

\node[fill=black!100, state, scale=0.10, vrtx/.style args = {#1/#2}{label=#1:#2}] (2) [vrtx=right/$a_4$] [ below right of = 1] {};

\node[fill=black!100, state, scale=0.10, vrtx/.style args = {#1/#2}{label=#1:#2}] (3) [vrtx=left/$a_5$] [ below left of = 1] {};

\node[fill=black!100, state, scale=0.10, vrtx/.style args = {#1/#2}{label=#1:#2}] (4) [vrtx=right/$a_3$] [ below right of = 2] {};

\node[fill=black!100, state, scale=0.10, vrtx/.style args = {#1/#2}{label=#1:#2}] (5) [vrtx=below/$a_2$] [ below left of = 2] {};

\node[fill=black!100, state, scale=0.10, vrtx/.style args = {#1/#2}{label=#1:#2}] (6) [vrtx=left/$a_1$] [ below left of = 3] {};

\draw[edge] (6) -- (5) node[midway, below] {$e_{1,2}$};
\draw[edge] (5) -- (3) node[midway, left] {$e_{2,5}$};
\draw[edge] (5) -- (2) node[midway, right] {$e_{2,4}$};
\draw[edge] (5) -- (4) node[midway, below] {$e_{2,3}$};
\draw[edge] (2) -- (1) node[midway, right] {$e_{4,6}$};
\draw[edge] (6) -- (3);
\draw[edge] (3) -- (2);
\draw[edge] (3) -- (1);
\draw[edge] (4) -- (2);
\end{tikzpicture}
\caption{}
\label{fig:exam 1}
\end{figure}








Rename the edges $e_{1,2}, \, e_{2,5}, \, e_{2,4}, \, e_{2,3}, \,   e_{4,6}$ as $v_1, \, v_2, \, v_{3}, \, v_{4}, \, v_{5}$ respectively and let $S = \{ v_1, \, v_2, \, v_{3}, \, v_{4}, \, v_{5}\}$.
First we construct $\widetilde{\Gamma}$ as depicted in the proof of Theorem~\ref{thm: bb grp LOG}. The vertex set of $\widetilde{\Gamma}$ is $V\widetilde{\Gamma} = S = \{ v_1, \, v_2, \, v_{3}, \, v_{4}, \, v_{5}\}$. The loops of $\widetilde{\Gamma}$ are the following: at the vertex $v_i$ one loop labelled by $v_{i+1}$ for $i= 1, 2, 3$; and at $v_5$ one loop labelled by $v_2^{-1}v_3$ (see the Figure~\ref{fig: wlog BB} below).

\begin{figure}[H]
    \centering
\begin{tikzpicture}[shorten >=1pt,node distance=2.5cm,auto]
\tikzset{
    edge/.style={draw=black,postaction={on each segment={mid arrow=black}}}
}
\node [state,scale=0.70] (1) {$v_1$};
\node[state, scale=0.68] (2) [ right of = 1] {$v_2$};
\node[state, scale=0.68] (3) [ right of = 2] {$v_3$};
\node[state, scale=0.68] (4) [ right of = 3] {$v_4$};
\node[state, scale=0.68] (5) [ right of = 4] {$v_5$};
\draw [->] (1) to [out=135,in=45,looseness=8] node[above] {$v_2$} (1);
\draw [->] (2) to [out=135,in=45,looseness=8] node[above] {$v_3$} (2);
\draw [->] (3) to [out=135,in=45,looseness=8] node[above] {$v_4$} (3);
\draw [->] (5) to [out=135,in=45,looseness=8] node[above] {$v_2^{-1}v_3$} (5);
\end{tikzpicture}
\caption{}
\label{fig: wlog BB}
\end{figure}

Following the notation as stated in Theorem~\ref{thm: Schur of BB}, $H_\Gamma \cong F_5/N$, where $F_5 = \langle v_1, \, v_2, \, v_{3}, \, v_{4}, \, v_{5} \rangle$ and $N =\left\langle \! \left\langle [v_{1}, v_{2}],\, [v_{2}, v_{3}],\, [v_{3}, v_{4}], \, [v_5, v_2^{-1}v_3]\right\rangle \! \right\rangle_{F_5}$. $H_2(H_\Gamma)$ is isomoprphic to $\ZZ^4$ with a basis given by the $\psi_S$-images of the commutators $[v_i, v_{i+1}]$ for $i= 1, 2, 3$ and $[v_5, v_2^{-1}v_3]$. On the other hand, it is straight forward that $\widetilde{B}_0(H_\Gamma) \cong N/\langle K(F_5) \cap N \rangle$ is trivial.


\end{exam}

\begin{exam}
    Let $A = \langle a_1, a_2, a_3 \mid a_1 (a_2 a_1)^2 = (a_2 a_1)^2 a_2, \, a_1(a_3a_1) = (a_3a_1)a_3, \, a_2(a_3a_2)^3 = (a_3a_2)^3 a_3\rangle$. The corresponding Artin-Tits system $\Gamma$ is as depicted in Figure~\ref{fig:exam 2}.

\begin{figure}[H]
    \centering
\begin{tikzpicture}[shorten >=1pt,node distance=20cm,auto]
\tikzset{
    edge/.style={draw=black,postaction={on each segment={mid arrow=black}}}
}


\node[fill=black!100, state, scale=0.10, vrtx/.style args = {#1/#2}{label=#1:#2}] (1) [vrtx=left/$a_3$] {};

\node[fill=black!100, state, scale=0.10, vrtx/.style args = {#1/#2}{label=#1:#2}] (2) [vrtx=below/$a_2$] [ below right of = 1] {};

\node[fill=black!100, state, scale=0.10, vrtx/.style args = {#1/#2}{label=#1:#2}] (3) [vrtx=below/$a_1$] [ below left of = 1] {};



\draw[edge] (3) -- (2) node[midway, below] {$5$};
\draw[edge] (3) -- (1) node[midway, left] {$3$};
\draw[edge] (2) -- (1) node[midway, right] {$7$};
\end{tikzpicture}
\caption{}
\label{fig:exam 2}
\end{figure}

According to Theorem~\ref{thm: multiplier vs Artin}, we draw the WLOG $\widetilde{\Gamma}$ (see Figure~\ref{fig: wlog A}) so that $A \cong G(\widetilde{\Gamma})$. 

\begin{figure}[H]
    \centering
\begin{tikzpicture}[shorten >=1pt,node distance=20cm,auto]
\tikzset{
    edge/.style={draw=black,postaction={on each segment={mid arrow=black}}}
}


\node[fill=black!100, state, scale=0.10, vrtx/.style args = {#1/#2}{label=#1:#2}] (1) [vrtx=left/$a_3$] {};

\node[fill=black!100, state, scale=0.10, vrtx/.style args = {#1/#2}{label=#1:#2}] (2) [vrtx=below/$a_2$] [ below right of = 1] {};

\node[fill=black!100, state, scale=0.10, vrtx/.style args = {#1/#2}{label=#1:#2}] (3) [vrtx=below/$a_1$] [ below left of = 1] {};



\draw[edge] (3) -- (2) node[midway, below] {$(a_2a_1)^2$};
\draw[edge] (3) -- (1) node[midway, left] {$(a_3a_1)$};
\draw[edge] (2) -- (1) node[midway, right] {$(a_3a_2)^3$};
\end{tikzpicture}
\caption{}
\label{fig: wlog A}
\end{figure}

Clearly, $\widetilde{\Gamma}$ is connected and let us choose $a_1$ as the base vertex for the fundamental group $\pi_1(\widetilde{\Gamma},a_1)$. We choose the spanning tree $\mathcal{T}$ of $\widetilde{\Gamma}$ having $E\mathcal{T} = \{ e_{1,2}, e_{2,3} \}$, i.e., the edges labelled by $(a_2a_1)^2$ and $(a_3a_2)^3$ respectively. Hence, $\mathcal{G} = \langle a_1, a_2, a_3 \mid a_1 (a_2 a_1)^2 = (a_2 a_1)^2 a_2, \,  a_2(a_3a_2)^3 = (a_3a_2)^3 a_3\rangle$ and $A \cong \mathcal{G}/N$, where $N = \left\langle \! \left\langle a_1 (a_3 a_1) = (a_3 a_1) a_3\right\rangle \! \right\rangle_{\mathcal{G}}$. So, it follows that $\widetilde{B_0}(A) \cong N/\langle K(\mathcal{G} \cap N)\rangle \cong \ZZ$. Note that, in this example $\widetilde{B_0}(A)$ is cyclic as there is only one edge (non loop) outside the spanning tree $\mathcal{T}$. In other words, $\widetilde{B_0}(A) \cong \ZZ$ and is generated by the $\Psi_{a_1}$-image of the relation $a_1(a_3a_1)=(a_3a_1)a_3$. Similarly, $H_2(A) \cong \ZZ$.

\end{exam}

\section*{Acknowledgements}

\noindent The author is grateful to Manoj Kumar Yadav for bringing to attention the Bogomolov multiplier. 
The author also expresses gratitude for the generous hospitality received from Harish-Chandra Research Institute.

\end{document}